\documentclass[11pt]{article}

\usepackage{amsmath}
\usepackage{amsthm}
\usepackage{amsfonts}
\usepackage{amscd}
\usepackage{amssymb}
\usepackage{graphicx,times}
\usepackage{mathtools}
\usepackage{url}
\usepackage{fullpage}
\usepackage{natbib}

\DeclarePairedDelimiter\abs{\lvert}{\rvert}

\newcommand{\R}{\mathbb{R}}

\newcommand{\betastar}{\hat{\beta}^{\textstyle{*}}}
\newcommand{\betahat}{\hat{\beta}}

\newcommand{\betastarT}{\hat{\beta}^{\textstyle{*}^T}}
\newcommand{\betabar}{\bar{\beta}^{\textstyle{*}}}

\newcommand{\epstar}{\varepsilon^{\textstyle{*}}}
\newcommand{\epstarT}{\varepsilon^{{\textstyle{*}^T}}}
\newcommand{\epresstar}{\widehat{\varepsilon}^{\textstyle{*}}}
\newcommand{\epresstarT}{\widehat{\varepsilon}^{\textstyle{*}^T}}

\newcommand{\Xstar}{X^{\textstyle{*}}}
\newcommand{\Wstar}{W^{\textstyle{*}}}
\newcommand{\Wstarinv}{W^{{\textstyle{*}^{-1}}}}
\newcommand{\Zstar}{Z^{\textstyle{*}}}
\newcommand{\YstarT}{Y^{{\textstyle{*}^T}}}
\newcommand{\XstarT}{X^{\textstyle{*}^T}}
\newcommand{\Sigstar}{\widehat{\Sigma}^{\textstyle{*}}}

\newcommand{\Lnorm}{\mathcal{L}}

\newcommand{\Proj}{\mathcal{P}}

\newcommand{\X}{\mathbb{X}}
\newcommand{\Y}{\mathbb{Y}}
\newcommand{\Ymatstar}{\Y^{\textstyle{*}}}
\newcommand{\YmatstarT}{\Y^{\textstyle{*}^T}}
\newcommand{\Xmatstar}{\X^{\textstyle{*}}}
\newcommand{\XmatstarT}{\X^{\textstyle{*}^T}}

\newcommand{\vecop}[1]{\text{vec}\left( #1 \right)}
\newcommand{\vechop}[1]{\text{vech}\left( #1 \right)}

\DeclareMathOperator{\tr}{tr}
\DeclareMathOperator{\Var}{var}

\newtheorem{lem}{Lemma} 
\newtheorem{thm}{Theorem} 
\allowdisplaybreaks

\begin{document}



\begin{small}
\begin{center}

\noindent {\LARGE \bf Bootstrapping for multivariate linear regression models}

\vspace{.1in}

Daniel J. Eck 

\vspace{.01in}

\textit{Department of Biostatistics, Yale School of Public Health.} \\
\textit{daniel.eck@yale.edu}

\end{center}
\end{small}

\begin{abstract}
The multivariate linear regression model is an important tool for 
investigating relationships between several response variables and several 
predictor variables. The primary interest is in inference about 
the unknown regression coefficient matrix. We propose multivariate bootstrap 
techniques as a means for making inferences about the unknown regression 
coefficient matrix. These bootstrapping techniques are extensions of those 
developed in \cite{freedman}, which are only appropriate for univariate 
responses. Extensions to the multivariate linear regression model are
made without proof. We formalize this extension and prove its validity. 
A real data example and two simulated data examples which offer some finite 
sample verification of our theoretical results are provided.
\end{abstract}


\section{Introduction}

The linear regression model is an important and useful tool in many 
statistical analyses for studying the relationship among variables.  
Regression analysis is primarily used for predicting values of the response 
variable at interesting values of the predictor variables, discovering the 
predictors that are associated with the response variable, and estimating 
how changes in the predictor variables affects the response variable 
\citep*{weisberg}. The standard linear regression methodology assumes that 
the response variable is a scalar. However, it may be the case that one is 
interested in investigating multiple response variables simultaneously. 
One could perform a regression analysis on each response separately in this 
setting. Such an analysis would fail to detect associations between responses. 
Regression settings where associations of multiple responses is of interest 
require a multivariate linear regression model for analysis. 

Bootstrapping techniques are well understood for the linear regression 
model with a univariate response \citep*{bickel, freedman}. In particular, 
theoretical justification for the residual bootstrap as a way to estimate the 
variability of the ordinary least squares (OLS) estimator of the regression 
coefficient vector in this model has been developed \citep*{freedman}. 
Theoretical extensions of residual bootstrap techniques appropriate for the 
multivariate linear regression model have not been formally introduced. The 
existence of such an extension is stated without proof and rather implicitly 
in subsequent works \citep*{freedman84, diaconis}. In this article we show 
that the bootstrap procedures in \cite{freedman} provide consistent estimates 
of the variability of the OLS estimator of the regression coefficient matrix 
in the multivariate linear regression model. Our proof technique follows 
similar logic as \cite{freedman}. The generality of the bootstrap theory 
developed in \cite{bickel} provide the tools required for our extension to 
the multivariate linear regression model.

\section{Bootstrap for the multivariate linear regression model}

The multivariate linear regression is
\begin{equation}\label{MVmodel}
Y_i = \beta X_i + \varepsilon_i, \quad (i = 1,...,n),
\end{equation}
where $Y_i \in \R^r$ and $r >1$ in order to have an interesting problem, 
$\beta \in \R^{r\times p}$, $X_i \in \R^p$, and the 
$\varepsilon_i's \in \R^r$ are errors having mean zero and 
variance-covariance matrix $\Sigma$ where $\Sigma > 0$. 
It is assumed that separate realizations from the model \eqref{MVmodel} are 
independent and that $n > p$. We further define $\X \in \R^{n\times p}$ as the 
design matrix with rows $X_i^T$, 
$\Y \in \R^{n \times r}$ is the matrix of responses with rows $Y_i^T$, and 
$\varepsilon \in \R^{n\times r}$ is the matrix of all errors with rows 
$\varepsilon_{i}^{T}$. The OLS estimator of $\beta$ in model \eqref{MVmodel} 
is $\hat{\beta} = \Y^T\X(\X^T\X)^{-1}$. We let 
$\widehat{\varepsilon} \in \R^{n\times r}$ denote the matrix of residuals 
consisting of rows $\widehat{\varepsilon}_i^{T} = (Y_i - \hat{\beta}X_i)^T$. 
The multivariate linear regression model assumed here is slightly 
different than the traditional multivariate linear regression model. The 
traditional model makes the additional assumptions that the errors are 
normally distributed and the design matrix $\X$ is fixed. 

We consider two bootstrap procedures that consistently estimate the 
asymptotic variability of $\text{vec}(\hat{\beta})$ under different 
assumptions placed upon the model \eqref{MVmodel}, where the $\text{vec}$ 
operator stacks the columns of a matrix so that 
$\text{vec}(\hat{\beta}) \in \R^{rp \times 1}$. The first bootstrap 
procedure is appropriate when the design matrix $\X$ is assumed to be fixed 
and the errors are constant. In this setup, residuals are resampled. The 
second bootstrap procedure is appropriate when $(X_i^T,\varepsilon_i^T)^T$ are 
realizations from a joint distribution. In this setup, cases $(X_i^T,Y_i^T)^T$ 
are resampled. It is known that bootstrapping under these setups provides a 
consistent estimator of the variability of $\Var(\hat{\beta})$ in 
model \eqref{MVmodel} when $r = 1$ \citep*{freedman}. We now provide the needed 
extensions. 


\subsection{Fixed design}

We first establish the residual bootstrap of \cite{freedman} when $\X$ is 
assumed to be a fixed design matrix.
Resampled, starred, data is generated by the model
\begin{equation} \label{starresid}
  \Ymatstar = \X\hat{\beta}^T + \epstar,
\end{equation}
where $\epstar \in \R^{n\times r}$ is the matrix of errors with rows being 
independent. The rows in $\epstar$ have common distribution 
$\widehat{F}_n$ which is the empirical distribution of the residuals from the 
original dataset, centered at their mean. Now 
$\betastar = \YmatstarT\X(\X^T\X)^{-1}$ is the OLS estimator of $\beta$ from 
the starred data. This process is performed a total 
of $B$ times with a new estimator $\betastar$ computed from \eqref{starresid} 
at each iteration. We then estimate the variability of 
$\text{vec}(\hat{\beta})$ with
$$
  \Var^{\textstyle{*}}\left\{\text{vec}(\hat{\beta})\right\} 
    = (B-1)^{-1}\sum_{b=1}^B\left\{\text{vec}(\betastar_b) 
      - \text{vec}(\betabar)\right\}
    \left\{\text{vec}(\betastar_b) 
      - \text{vec}(\betabar)\right\}^T
$$
where $\betastar_b$ is the residual bootstrap estimator of $\beta$ at 
iteration $b$ and $\betabar = B^{-1}\sum_{b=1}^B\betastar_b$. We summarize 
this bootstrap procedure in Algorithm 1.

\begin{center}
{\bf Algorithm 1.} Bootstrap procedure with fixed design matrix. \\
\begin{itemize}
  \item[Step 1.] Set $B$ and initialize $b = 1$.
  \item[Step 2.] Sample residuals from $\widehat{F}_n$, with replacement, 
    and compute $\Ymatstar$ as in \eqref{starresid}.
  \item[Step 3.] Compute $\betastar_b = \YmatstarT\X(\X^T\X)^{-1}$, store 
    $\text{vec}(\betastar_b)$, and let $b = b+1$.
  \item[Step 4.] Repeat Steps 2-3, iterating $b$ before returning to Step 2. 
  \item[Step 5.] When $b = B$, compute 
    $\Var^{\textstyle{*}}\left\{\text{vec}(\hat{\beta})\right\}$.
\end{itemize}
\end{center}

\vspace*{0.5cm}

Before the theoretical justification of the residual bootstrap is formally 
given, some important quantities are stated.
The residuals from the regression \eqref{starresid} are 
$\epresstar = \Ymatstar - \X\betastarT$. The variance-covariance matrix $\Sigma$ 
in model \eqref{MVmodel} is then estimated by 
$$
  \widehat{\Sigma} = n^{-1}\sum_{i=1}^n \widehat{\varepsilon}_i
    \widehat{\varepsilon}_i^T - \hat{\mu}^2, \qquad
  \hat{\mu}^2 = \left(n^{-1}\sum_{i=1}^n\widehat{\varepsilon}_i\right)
    \left(n^{-1}\sum_{i=1}^n\widehat{\varepsilon}_i\right)^T.
$$
Likewise, the variance-covariance estimate from the starred data is 
$$
  \Sigstar = n^{-1}\sum_{i=1}^n \epresstar_i\epresstarT_i 
    - \hat{\mu}^{\textstyle{*}^2}, \qquad 
  \hat{\mu}^{\textstyle{*}^2} = 
    \left(n^{-1}\sum_{i=1}^n\epresstar_i\right)
    \left(n^{-1}\sum_{i=1}^n\epresstar_i\right)^T.
$$
Let $I_k$ denote the $k \times k$ identity matrix. Theorem 1 provides 
bootstrap asymptotics for the regression model \eqref{MVmodel}. It extends 
Theorem 2.2 of \cite{freedman} to the multivariate setting.

\begin{thm} \label{MVboot-MVLR}
Assume the regression model \eqref{MVmodel} where the errors have finite 
fourth moments. Suppose that $n^{-1}\X^T\X \to \Sigma_{X} > 0$. Then, 
conditional on almost all sample paths $Y_1,...,Y_n$, as $n \to \infty$,
\begin{itemize}
\item[a)] 
$
  \surd{n}\left\{\text{vec}(\hat{\beta}^{\textstyle{*}}) 
  - \text{vec}(\hat{\beta})\right\} 
  \to^d 
  N(0, \Sigma_{X}^{-1} \otimes \Sigma)
$,
\item[b)] $\Sigstar \to_p\Sigma$, and
\item[c)] 
$
  \left\{ (\X^T\X)^{1/2} \otimes \widehat{\Sigma}^{\textstyle{*}^{-1/2}}
  \right\}\left\{\text{vec}(\betastar) - \text{vec}(\hat{\beta})\right\}
  \to^d 
  N(0,I_{rp})
$
\end{itemize}
\end{thm}

The proof of Theorem~\ref{MVboot-MVLR}, along with the details of several 
necessary lemmas and theorems, are included in the theoretical details 
section. Theorem~\ref{MVboot-MVLR} establishes the multivariate analogue for 
the residual bootstrap. This theorem shows that standard error estimation of 
the estimated $\beta$ matrix obtained through bootstrapping, is 
$\surd{n}$-consistent. Now let $f:\R^{rp}\to\R^k$ be a differentiable 
function. Then the conclusions of Theorem~\ref{MVboot-MVLR} can be applied to 
establish a multivariate delta method based on estimates obtained via the 
residual bootstrap. This immediately follows from a first order Taylor 
expansion and some algebra arriving at
\begin{equation} \label{MVdelta}
  \surd{n}\left[ f\left\{\text{vec}(\hat{\beta}^{\textstyle{*}})\right\} 
      - f\left\{\text{vec}(\hat{\beta})\right\}\right] 
  = \nabla f\left\{\text{vec}(\hat{\beta})\right\} 
      \surd{n}\left\{\text{vec}(\hat{\beta}^{\textstyle{*}}) 
      - \text{vec}(\hat{\beta})\right\} + O_{p}(n^{-1/2}).
\end{equation}
Therefore \eqref{MVdelta} converges weakly to a normal distribution with mean 
zero and variance given by 
$$
  \nabla f\left\{\vecop{\beta}\right\} \left(\Sigma_{X}^{-1} 
  \otimes \Sigma\right) \nabla^T f\left\{\vecop{\beta}\right\}
$$ 
as $n\to\infty$.

\subsection{Random design and heteroskedasticity}

In this section we assume that the $X_i$s in model \eqref{MVmodel} are 
realizations of a random variable $X$. The regression coefficient matrix 
$\beta$ now takes the form 
$
  \beta = E(YX^T)\Sigma_X^{-1}
$
where $\Sigma_X = E(XX^T)$ and it is assumed that $\Sigma_X > 0$. 
Now that $X$ is stochastic, there may be some association 
between $X$ and the errors $\varepsilon$. The possibility of 
heteroskedasticity means that we need to alter the bootstrap procedure 
outlined in the previous section in order to consistently estimate the 
variability of $\text{vec}(\betahat)$.

It is assumed that the data vectors $(X_i^T, Y_i^T)^T \in \R^{p+r}$ are  
independent, with a common distribution $\mu$ and 
$E(\|(X_i^T, Y_i^T)^T\|^4) < \infty$ where $\|\cdot\|$ is the Euclidean norm. 
Unlike the fixed design setting, data pairs $(X_i^T, Y_i^T)^T$ are resampled 
with replacement to form the starred data $(\XstarT_i, \YstarT_i)^T$, 
for $i =1,...,n$. Given the original sample, 
$(X_i^T, Y_i^T)^T$, $i = 1,...,n$, the resampled vectors are independent, 
with distribution $\mu_n$. Denote $\Xmatstar \in \R^{n\times p}$ 
and $\Ymatstar \in \R^{n\times r}$ as the matrix with rows $\XstarT_i$ and 
$\YstarT_i$ respectively. The starred estimator of $\beta$ obtained from 
resampling is then 
$
  \betastar = \YmatstarT\Xmatstar\left(\XmatstarT\Xmatstar\right)^{-1}.
$
For every $n$ there is positive probability, albeit low, that 
$\XmatstarT\Xmatstar$ is singular, and the probability of singularity decreases 
exponentially in $n$. We assume that displayed equation (1.17) in 
\cite{chatterjee} holds in order to circumvent singularity in our bootstrap 
procedure.

The bootstrap is performed a total of $B$ times with a new estimator 
$\betastar$ computed at each iteration. We then estimate the variability of 
$\text{vec}(\hat{\beta})$ with
$$
  \Var^{\textstyle{*}}\left\{\text{vec}(\hat{\beta})\right\} 
    = (B-1)^{-1}\sum_{b=1}^B\left\{\text{vec}(\betastar_b) 
      - \text{vec}(\betabar)\right\}
    \left\{\text{vec}(\betastar_b) 
      - \text{vec}(\betabar)\right\}^T
$$
where $\betastar_b$ is the bootstrap estimator of $\beta$ at iteration $b$ 
and $\betabar = B^{-1}\sum_{b=1}^B\betastar_b$. We summarize this bootstrap 
procedure in Algorithm 2.

\begin{center}
{\bf Algorithm 2.} Bootstrap procedure with random design matrix. \\
\begin{itemize}
  \item[Step 1.] Set $B$ and initialize $b = 1$.
  \item[Step 2.] Resample $(X_i^T, Y_i^T)^T$ with replacement.
  \item[Step 3.] Compute 
  $
    \betastar_b = \YmatstarT\Xmatstar(\XmatstarT\Xmatstar)^{-1}
  $, store $\text{vec}(\betastar_b)$.
  \item[Step 4.] Repeat Steps 2-3, iterating $b$ before returning to Step 2.
  \item[Step 5.] When $b = B$, compute 
    $\Var^{\textstyle{*}}\left\{\text{vec}(\hat{\beta})\right\}$.
\end{itemize}
\end{center}

\vspace*{0.5cm}

We now show that the variability of $\text{vec}(\betahat)$ is estimated 
consistently by our multivariate bootstrap procedure which resamples cases. 
Let $M$ be a non-negative definite matrix with entries 
$
  M_{jk} = E\left\{\text{vec}(X_i\varepsilon_i^T)_j
    \text{vec}(X_i\varepsilon_i^T)_k\right\}
$
for $j,k = 1,...,rp$ and define
$
  \Delta = \left(\Sigma_X^{-1} \otimes I_r\right)M
    \left(\Sigma_X^{-1} \otimes I_r\right).
$
where $n^{-1}\X^T\X \to \Sigma_X$ a.e. as $n\to\infty$. Then
\begin{equation} \label{asymp}
  \begin{split}
    \sqrt{n}\text{vec}\left(\betahat - \beta\right) 
      &= \text{vec}\left\{
        n^{-1/2}\varepsilon^T\X(n^{-1}\X^T\X)^{-1}
        \right\} \\
      &= \left\{(n^{-1}\X^T\X)^{-1} \otimes I_r\right\}
        \text{vec}\left(n^{-1/2}\varepsilon^T\X\right)
      \to N(0, \Delta).  
  \end{split}    
\end{equation}
The next theorem states that 
$\sqrt{n}\text{vec}\left(\betastar - \betahat\right)$ is the same 
as \eqref{asymp}. This is an extension of Theorems 3.1 and 3.2 of 
\cite{freedman} to the multivariate linear regression setting.

\begin{thm} \label{MVboot-MVLR-randomX}
Assume that $(X_i^T, Y_i^T)^T \in \R^{p+r}$ are independent, with a common 
distribution $\mu$, $E(\|(X_i^T,Y_i^T)^T\|^4) < \infty$, and 
$\Sigma_X = E(XX^T)$ is positive definite. Then, conditional on almost all 
sample paths, $(X_i^T, Y_i^T)^T$, $i = 1,...,n$, as $n\to\infty$,
\begin{itemize}
  \item[a)] $n^{-1}\left(\XmatstarT\Xmatstar\right) \to_p\Sigma_X$, 
  \item[b)] 
  $
    \sqrt{n}\left\{\text{vec}(\betastar) - \text{vec}(\betahat)\right\}
    \to^d
    N(0, \Delta)
  $, and
  \item[c)] the sequence $\Sigstar \to_p \Sigma$.
\end{itemize}
\end{thm}

The proof of Theorem~\ref{MVboot-MVLR-randomX}, along with necessary lemmas, 
are included in the theoretical details section.

\section{Examples}

\subsection{Simulations}

In this section we provide two simulated examples which show support for our 
multivariate bootstrap procedures.

\subsubsection{Fixed design}

This example illustrates Theorem~\ref{MVboot-MVLR}. We generated 
data according to the multivariate linear regression model~\eqref{MVmodel} 
where $Y_i \in \R^3$, $X_i \in \R^2$, and both $\beta$ and $\Sigma$ are 
prespecified. Our goal is to make inference about $\text{vec}(\beta)$ using 
confidence regions. For each component of $\beta$, a 95\% percentile interval 
computed using the residual bootstrap in Algorithm 1 is compared with a 95\% 
confidence interval that assumes model \eqref{MVmodel} is correct. Four data 
sets were generated at different sample sizes and the performance of the 
multivariate residual bootstrap is assessed. The bootstrap is performed 
$B = 4n$ times in each dataset. The results are displayed in 
Table~\ref{boot-ex1-tab1}. For the first two components of $\beta$, we see 
that the confidence regions obtained from both methods are close to each 
other and that the distance between the two shrinks as $n$ increases. 
Similar results are obtained for the other components of $\beta$.

\begin{table}[h!]
\begin{center}
\begin{tabular}{llcc}
  & component & bootstrap & confidence \\
  \hline
  $n = 100$  & $\text{vec}(\beta)_1$ & (-0.062 0.958) & (-0.092 0.997) \\
             & $\text{vec}(\beta)_2$ & (-0.873 0.330) & (-0.922 0.342) \\
             \hline
  $n = 500$  & $\text{vec}(\beta)_1$ & ( 0.279 0.826) & (0.256 0.823) \\
             & $\text{vec}(\beta)_2$ & ( 0.070 0.655) & (0.074 0.658) \\
             \hline
  $n = 1000$ & $\text{vec}(\beta)_1$ & ( 0.415 0.771) & ( 0.410 0.768) \\
             & $\text{vec}(\beta)_2$ & (-0.010 0.364) & (-0.020 0.350) \\
             \hline
  $n = 5000$ & $\text{vec}(\beta)_1$ & ( 0.509 0.684) & ( 0.509 0.684) \\
             & $\text{vec}(\beta)_2$ & (-0.031 0.143) & (-0.030 0.143)
\end{tabular}
\end{center}
\caption{ Comparison of the 95\% percentile interval and a 95\% confidence 
  interval for the first two components of $\text{vec}(\beta)$. The number 
  of bootstrap samples is $B = 4n$ for each dataset. }
\label{boot-ex1-tab1}
\end{table}

\subsubsection{Random design and heteroskedasticity}

This example aims to show support for Theorem~\ref{MVboot-MVLR-randomX}. We 
generated data according to the multivariate linear regression 
model~\eqref{MVmodel} where $Y_i \in \R^3$, $X_i \in \R^2$, and both $\beta$ 
and $\Sigma$ are prespecified. The predictors and errors are generated 
according to
$$
  \left(\begin{array}{c}
    X_i \\
    \varepsilon_i
  \end{array}\right) \sim
  N\left\{\left(\begin{array}{c}
    0 \\
    0
  \end{array}\right),\left(\begin{array}{cc}
    \Sigma_X & \Sigma_{X \varepsilon} \\
    \Sigma_{\varepsilon X} & \Sigma
  \end{array}\right)
  \right\},
$$
for $i = 1,...,n$. Our goal is to make inference about $\text{vec}(\beta)$ 
using the multivariate bootstrap procedure in the random design case. 
For each component of $\beta$, a 95\% percentile interval 
computed using the residual bootstrap in Algorithm 2 is compared with a 95\% 
confidence interval that assumes model \eqref{MVmodel} with heterogeneity is 
correct. Three data sets were generated at different sample sizes and the 
performance of the multivariate bootstrap is assessed. The bootstrap is 
performed a total of $B = 4n$ times in each dataset. The results are 
displayed in Table~\ref{boot-ex2-tab1}. For the first two components of 
$\beta$, we see that the confidence regions obtained from both methods are 
close to each other and that the distance between the two shrinks as $n$ 
increases. Similar results are obtained for the other components of $\beta$.


\begin{table}[h!]
\begin{center}
\begin{tabular}{llcc}
   & component & bootstrap & confidence \\
   \hline
   $n = 100$  & $\text{vec}(\beta)_1$ & (-0.013 1.617) & (0.205 1.391) \\
              & $\text{vec}(\beta)_2$ & ( 0.232 1.438) & (0.296 1.366) \\
              \hline
   $n = 500$  & $\text{vec}(\beta)_1$ & ( 0.638 1.208) & (0.646 1.198) \\
              & $\text{vec}(\beta)_2$ & ( 0.329 0.912) & (0.369 0.868) \\
              \hline
   $n = 1000$ & $\text{vec}(\beta)_1$ & ( 0.937 1.323) & (0.952 1.304) \\
              & $\text{vec}(\beta)_2$ & ( 0.646 0.987) & (0.659 0.982) \\
              \hline
   $n = 5000$ & $\text{vec}(\beta)_1$ & ( 0.995 1.161) & (0.997 1.160) \\
              & $\text{vec}(\beta)_2$ & ( 0.608 0.771) & (0.616 0.764)
  \end{tabular}
\end{center}
\caption{ Comparison of the 95\% percentile interval and a 95\% confidence 
  interval for the first two components of $\text{vec}(\beta)$. The number 
  of bootstrap samples is $B = 4n$ for each dataset. }
\label{boot-ex2-tab1}
\end{table}

\subsection{Cars data}


The data in this example, analyzed in \cite{henderson}, was extracted from the 
1974 Motor Trend US magazine. The objective of this study is to compare 
aspects of automobile design on performance and fuel composition for 32 
automobiles (1973-74) models. In this analysis, we assume that the multivariate 
model~\eqref{MVmodel} with miles per gallon, displacement, and horse power as 
response variables and number of cylinders and transmission type are predictors. 
Number of cylinders and transmission type are both factor variables. The 
automobiles have either 4, 6, or 8 cylinders and their transmission type is 
either automatic or manual.

For inference for $\beta$, we compare a 95\% bootstrap percentile region using 
the fixed design bootstrap in Algorithm 1 with a 95\% confidence 
interval. The number of bootstrap resamples is set at $B = 4n$. The results 
are depicted in Table~\ref{cars-data-tab1}. We see that inferences about 
$\beta$ are fairly similar for both methods.

\begin{table}[h!]
\begin{center}
\begin{tabular}{lll}
  component & bootstrap & confidence \\
  \hline
$\text{vec}(\beta)_1$ & (   2.734   7.027) & (   2.286   7.136) \\
$\text{vec}(\beta)_2$ & (  -3.693   0.630) & (  -3.806   0.916) \\
$\text{vec}(\beta)_3$ & (  -6.823  -4.173) & (  -6.900  -3.812) \\
$\text{vec}(\beta)_4$ & (   0.326   5.745) & (   0.181   4.939) \\
$\text{vec}(\beta)_5$ & (-134.667 -52.921) & (-134.408 -56.787)
\end{tabular}
\end{center}
\caption{ Comparison of the 95\% percentile interval and a 95\% confidence 
  interval for the first five components of $\text{vec}(\beta)$. }
\label{cars-data-tab1}
\end{table}

\section{Theoretical details}

Before we present our proof of Theorems~\ref{MVboot-MVLR} 
and~\ref{MVboot-MVLR-randomX}, we motivate the Mallows metric as a central 
tool for our proof technique. 
The Mallows metric for probabilities in $\R^p$, relative to the Euclidean norm 
was the driving force needed to establish the validity of the residual 
bootstrap approximation in the context of 
univariate regression \citep*{bickel, freedman}. The Mallows metric, relative 
to the Euclidean norm, for two probability measures $\mu,\nu$ in $\R^p$, 
denoted $d_l^p(\mu,\nu)$, is 
$$
  d_l^p(\mu,\nu) = \inf_{U \sim \mu,V \sim \nu} E^{1/l}\left(\|U-V\|^l\right).
$$ 
Properties of the Mallows metric are developed for random variables on 
separable Banach spaces of finite dimension \citep*{bickel}. Since $\R^k$ 
is indeed a separable Banach space for a natural number $k$, the theory in 
\cite{bickel} applies to our case.
In the present article, we use the Mallows metric when $r > 1$ to prove that 
the residual bootstrap can be used to estimate the variability of 
$\text{vec}(\hat{\beta})$ consistently.

\subsection{Fixed design}

Let $\Psi_n(F)$ be the distribution function of 
$\surd{n}\left\{\text{vec}(\hat{\beta}) - \text{vec}(\beta)\right\}$ 
where $F$ is the law of the errors $\varepsilon$ so that $\Psi_n(F)$ is a 
probability measure on $\R^{rp}$. Let $G$ be an alternate law of the errors, 
where it is assumed that $G$ is mean-zero with finite variance $\Sigma_G > 0$. 
In applications, $G$ will be the centered empirical distribution of the 
residuals.

\begin{thm} \label{identity} 
$\left[d_2^{rp}\left\{\Psi_n(F),\Psi_n(G)\right\}\right]^2 
  \leq n r\tr\left\{(\X^T\X)^{-1}\right\}\left\{d_2^r(F,G)\right\}^2$.
\end{thm}

\begin{proof}
Let $A = \X(\X^T\X)^{-1}$. Then $\Psi_n(F)$ is the law of 
$\surd{n}\varepsilon_n^T(F)A$ where $\varepsilon_n(F)$ is the matrix with 
$n$ rows of independent random variables $\varepsilon$, having common law $F$. 
$\Psi_n(G)$ can be thought of similarly. Observe that $A^TA = (\X^T\X)^{-1}$. 
Then, from Lemma 8.9 in \cite{bickel}, we see that
\begin{align*}
  &\left[d_2^{rp}\left\{\Psi_n(F),\Psi_n(G)\right\}\right]^2 
    = \left(d_2^{rp}\left[\text{vec}\{\varepsilon_n^T(F)A\},\text{vec}\{\varepsilon_n^T(G)A\}\right]\right)^2 \\
  &\qquad= \left(d_2^{rp}\left[ (A^T \otimes I_r)\text{vec}\{\varepsilon_n^T(F)\}, (A^T \otimes I_r)\text{vec}\{\varepsilon_n^T(G)\} \right]\right)^2 \\
  &\qquad\leq n \tr\left\{ (A^T\otimes I_r)(A^T\otimes I_r)^T \right\} \left\{d_2^r(F,G)\right\}^2 
    = n \tr\left\{ (A^T\otimes I_r)(A\otimes I_r) \right\} \left\{d_2^r(F,G)\right\}^2 \\
  &\qquad= n \tr\left( A^TA \otimes I_r \right) \left\{d_2^r(F,G)\right\}^2 
    = n \tr\left\{ (\X^T\X)^{-1} \otimes I_r \right\} \left\{d_2^r(F,G)\right\}^2 \\
  &\qquad= n r \tr\left\{ (\X^T\X)^{-1} \right\} \left\{d_2^r(F,G)\right\}^2,
\end{align*}
which is our desired conclusion. 
\end{proof}

With Theorem~\ref{identity} we can bound the distance between the sample 
dependent distribution functions $\Psi_n(F)$ and $\Psi_n(G)$ by the distance 
between their underlying laws. As in \cite{freedman}, we proceed with $F_n$ 
as the empirical distribution function of $\varepsilon_1,...,\varepsilon_n$. 
Let $\widetilde{F}_n$ be the empirical distribution of the residuals 
$\widehat{\varepsilon}_1,...,\widehat{\varepsilon}_n$ from the original 
regression, and let $\widehat{F}_n$ be $\widetilde{F}_n$ centered at its mean 
$\hat{\mu} = n^{-1}\sum_{i=1}^n\widehat{\varepsilon}_i$. Since 
$\widehat{\varepsilon} = \mathbb{Y} - \X\hat{\beta}^T$, we have
$
  \widehat{\varepsilon} - \varepsilon = -\Proj\varepsilon
$  
where $\Proj$ is the projection into the column space of $\X$.

\begin{lem} \label{lem1}
$E^2\left\{d_2^r(\widetilde{F}_n,F_n)\right\} \leq p\tr(\Sigma)/n$.
\end{lem}

\begin{proof}
From the definition of the Mallows metric we have
\begin{align*}
\left\{d_2^r(\widetilde{F}_n,F_n)\right\}^2 &\leq n^{-1}\sum_{i=1}^n \|\widehat{\varepsilon}_i - \varepsilon_i\|^2 
  = n^{-1} \tr\left\{ \left(\widehat{\varepsilon} - \varepsilon\right)^T\left(\widehat{\varepsilon} - \varepsilon\right)\right\} \\
  &= n^{-1} \tr\left( \varepsilon^T \Proj \varepsilon \right).
\end{align*}
From linearity of the expectation with respect to the trace operator, 
$$
E \left\{\tr\left( \varepsilon^T \Proj \varepsilon\right)\right\} 
  = \tr\left\{ E \left(\varepsilon^T \Proj \varepsilon\right)\right\} 
  = \tr\left\{ \Proj E \left(\varepsilon \varepsilon^T\right) \right\} 
  \leq \tr\left( \Proj \right) \tr\left( \Sigma \right) 
  = p\tr\left( \Sigma \right)
$$
and this completes the proof. 
\end{proof}

\begin{lem} \label{lem2}
$E^2\left\{d_2^r(\widehat{F}_n,F_n)\right\} \leq (p + 1)\tr(\Sigma)/n$.
\end{lem}

\begin{proof}
From Lemma 8.8 in \cite{bickel} we have
\begin{align*}
d_2^r(\widehat{F}_n,F_n)^2 &= d_2^r\{\widetilde{F}_n - E(\widetilde{F}_n), F_n - E(F_n)\}^2 + \|E(F_n)\|^2 \\
  &= d_2^r(\widetilde{F}_n,F_n)^2 - \| E(\widetilde{F}_n) - E(F_n) \|^2 + \|E(F_n)\|^2 \\
  &\leq d_2^r(\widetilde{F}_n,F_n)^2 + \|n^{-1}\sum_{i=1}^n \varepsilon_i \|^2
\end{align*}
with the empirical distribution functions $F_n$,$\widetilde{F}_n$, and 
$\widehat{F}_n$ used as random variables in the application of Lemma 8.8 
in \cite{bickel}. We see that 
\begin{align*}
  E \left(\|n^{-1}\sum_{i=1}^n \varepsilon_i \|^2\right) 
    = n^{-2}\left\{E\left(\sum_{i=1}^n\varepsilon_i^T\varepsilon_i 
      + \sum_{i\neq j}\varepsilon_i^T\varepsilon_j\right)\right\} 
    = n^{-1}\left\{E (\varepsilon_1^T\varepsilon_1) \right\} 
    = n^{-1} \tr\left(\Sigma\right).
\end{align*}
Our conclusion follows from Lemma~\ref{lem1}.
\end{proof}

These results imply the validity of the bootstrap approximation 
for the model \eqref{MVmodel} if we assume that 
$n^{-1}\X^T\X \to \Sigma_{X} > 0$. From Theorem~\ref{identity},
$$  
  E\left[d_2^{rp}\{\Psi_n(\widehat{F}_n),\Psi_n(F)\}\right] 
    \leq  n r\tr\left\{(\X^T\X)^{-1}\right\}d_2^r(\widehat{F}_n,F)
$$  
and because of the metric properties of $d_2^r(\cdot,\cdot)$ 
$$
  \frac{1}{2}d_2^r(\widehat{F}_n,F)^2 
    \leq d_2^r(\widehat{F}_n,F_n)^2 + d_2^r(F_n,F)^2
$$
where Lemma~\ref{lem2} shows that 
$d_2^r(\widehat{F}_n,F_n)^2 \to_p 0$ and Lemma 8.4 of 
\citep*{bickel} implies that $d_2^r(F_n,F)^2 \to_p 0$ with the 
separable Banach space taken to be $\R^r$. The next results are special cases 
of \cite{lai} which are adapted from \cite{freedman} to the multivariate 
setting. We let $\varepsilon_j$, $j=1,...,r$, be the column of $\varepsilon$ 
corresponding to the errors of response $Y_j$.

\begin{lem} \label{lem3}
$n^{-1}\X^T\varepsilon \to 0 \;$ a.s. and 
$\hat{\beta} \to \beta \;$ a.s.
\end{lem}

\begin{proof}
Let $A_j$ be the $j$th column of $\varepsilon$. Then 
$n^{-1}\X^T\varepsilon \in \R^{p \times r}$ with columns 
$n^{-1}\X^T\varepsilon$. Lemma 2.3 of \cite{freedman} states that 
$n^{-1}\X^T A_j \to 0$ a.s. for any particular 
$j = 1,...,r$. Therefore $n^{-1}\X^T\varepsilon \to 0$ a.s.
A similar argument verifies our second result. 
\end{proof}

\begin{lem} \label{lem4}
$n^{-1}\tr\left\{ (\widehat{\varepsilon} - \varepsilon)^T(\widehat{\varepsilon} - \varepsilon) \right\} \to 0 \;$ a.s..
\end{lem}

\begin{proof}
A similar argument to that of Lemma 2.4 in \cite{freedman} gives
\begin{align*}
n^{-1}\tr\left\{ (\widehat{\varepsilon} - \varepsilon)^T(\widehat{\varepsilon} - \varepsilon) \right\} 
  &= n^{-1}\tr\left\{\varepsilon^T\X(\X^T\X)^{-1}\X^T\varepsilon \right\} \\
  &= \tr\left\{ \left(n^{-1}\varepsilon^T\X\right)\left(n^{-1}\X^T\X\right)^{-1}\left(n^{-1}\X^T\varepsilon\right) \right\}.
\end{align*}
The center term converges to $\Sigma_{X} > 0$ and the left and right terms converge to 0 a.s. by Lemma~\ref{lem3}. Our result follows.
\end{proof}

\begin{lem} \label{lem5}
$d_2^r(\widehat{F}_n,F_n) \to 0 \;$ a.s. and $d_2^r(\widehat{F}_n,F) \to 0 \;$ a.s.
\end{lem}

\begin{proof}
From the arguments in the proofs of Lemmas~\ref{lem1} and \ref{lem2} we have that 
\begin{align*}
d_2^r(\widehat{F}_n,F_n) &=  d_2^r(\widetilde{F}_n,F_n)^2 - \| E(\widetilde{F}_n) - E(F_n) \|^2 + \|E(F_n)\|^2 \\
  &= \|n^{-1}\sum_{i=1}^n\varepsilon_i\|^2 - \|n^{-1}\sum_{i=1}^n\left(\widehat{\varepsilon}_i - \varepsilon_i\right) \|^2 + d_2^r(\widetilde{F}_n,F_n) \\
  &\leq \|n^{-1}\sum_{i=1}^n\varepsilon_i\|^2 + n^{-1}\tr\left\{ (\widehat{\varepsilon} - \varepsilon)^T(\widehat{\varepsilon} - \varepsilon) \right\}
\end{align*}
which converges to 0 a.s. by Lemma~\ref{lem4}. Therefore the first convergence result holds. From the metric properties of the Mallows metric we have that 
$$
  \frac{1}{2}d_2^r(\widehat{F}_n,F)^2 \leq d_2^r(\widehat{F}_n,F_n)^2 + d_2^r(F_n,F)^2. 
$$
Our second convergence result follows from the first convergence result and Lemma 8.4 of \cite{bickel}.
\end{proof}

\begin{lem} \label{lem6}
Let $u_i$ and $v_i$, $i = 1,...,n$, be $r \times 1$ vectors. Let
$$
  \bar{u} = n^{-1}\sum_{i=1}^n u_i, \quad \text{and} \quad s^2_u = n^{-1}\sum_{i=1}^n(u_i - \bar{u})(u_i - \bar{u})^T  
$$  
and similarly for $v$. Then 
$$ 
  \|s_u^2 - s_v^2\|_F^2 \leq \|n^{-1}\sum_{i=1}^n(u_i - v_i)(u_i - v_i)^T\|_F^2 
$$  
where $\|\cdot\|_F$ is the Frobenius norm.
\end{lem}

\begin{proof}
We have 
\begin{align*}
\|s_u^2 - s_v^2\|_F^2 &= \sum_{j=1}^n\sum_{k=1}^n |n^{-1}\sum_{i=1}^n(u_i - \bar{u})_j(u_i - \bar{u})_k^T - 
  n^{-1}\sum_{i=1}^n(v_i - \bar{v})_j(v_i - \bar{v})_k^T |^2 \\
  &\leq \sum_{j=1}^n\sum_{k=1}^n |n^{-1}\sum_{i=1}^n(u_i - v_i)_j(u_i - v_i)_k|^2 \\
  &= \|n^{-1}\sum_{i=1}^n(u_i - v_i)(u_i - v_i)^T\|_F^2,
\end{align*}
where the inequality follows from \cite[Lemma 2.7]{freedman}.
\end{proof}

The proof of Theorem~\ref{MVboot-MVLR} is now given. Before we this Theorem, 
define the $\text{vech}(A) \in \R^{p(p+1)/2 \times 1}$ operator to be the 
function that stacks the unique $p(p+1)/2$ elements of any symmetric matrix 
$A \in \R^{p\times p}$.



\begin{proof}
Exchange $\widehat{F}_n$ for $G$ in Theorem~\ref{identity} and observe that 
$$
  d_2^{rp}\left\{\Psi_n(F),\Psi_n(\widehat{F}_n)\right\} 
    \leq nr\tr\left\{ (\X^T\X)^{-1}\right\} d_2^r(F,\widehat{F}_n)^2.
$$ 
From Lemma~\ref{lem5} we know that $d_2^r(F,\widehat{F}_n)^2 \to 0$ almost 
everywhere. Our result for part a) follows since $F$ is mean-zero normal with 
variance $\Sigma_{X}^{-1} \otimes \Sigma$. We now show that part b) holds. 
First, we need to establish that $\widehat{\Sigma} \to \Sigma$ almost 
everywhere. To see this, introduce 
$$  
  \Sigma_n = n^{-1}\sum_{i=1}^n \varepsilon_i\varepsilon_i^T - \left(n^{-1}\sum_{i=1}^n\varepsilon_i\right)\left(n^{-1}\sum_{i=1}^n\varepsilon_i\right)^T.
$$  
Clearly, $\Sigma_n \to \Sigma$ a.s. Let 
$C_n = n^{-1}\sum_{i=1}^n\left(\hat{\varepsilon}_i - \varepsilon_i\right)\left(\hat{\varepsilon}_i - \varepsilon_i\right)^T$. 
We have, 
\begin{align*}
  \|\widehat{\Sigma} - \Sigma_n\|_F^2 &\leq \|C_n\|_F^2 = \tr(C_n C_n) \leq \tr^2(C_n) \\
    &= \tr^2\left\{ n^{-1}\sum_{i=1}^n(\widehat{\varepsilon}_i - \varepsilon_i)^T(\widehat{\varepsilon}_i - \varepsilon_i)\right\} \\
    &= \tr^2\left\{ n^{-1}(\widehat{\varepsilon} - \varepsilon)^T(\widehat{\varepsilon} - \varepsilon) \right\} \to 0 
\end{align*}
a.s. where the first inequality follows from Lemma~\ref{lem6} 
with $\widehat{\Sigma}_n$ and $\Sigma_n$ taking the place of $s_u^2$ and 
$s_v^2$ respectively, the second inequality follows from the fact that $C_n$ 
is positive definite a.s., and the convergence follows from 
Lemma~\ref{lem4}.

Let 
$D_{n} = E\left(\|\Sigstar_n - \Sigma_n^{\textstyle{*}}\|_F \; \mid Y_1,...,Y_n\right)$. 
From Lemma~\ref{lem6} and the proof of Lemma~\ref{lem1} we see that,
\begin{align*}
D_{n} &\leq E\left\{\|n^{-1}\sum_{i=1}^n\left(\widehat{\varepsilon}^{\textstyle{*}}_i - \varepsilon^{\textstyle{*}}_i\right)\left(\widehat{\varepsilon}^{\textstyle{*}}_i - \varepsilon^{\textstyle{*}}_i\right)^T\|_F \; \mid Y_1,...,Y_n\right\} \\
  &\leq E\left[ \tr\left\{n^{-1}(\widehat{\varepsilon}^{\textstyle{*}} - \varepsilon^{\textstyle{*}})^T(\widehat{\varepsilon}^{\textstyle{*}} - \varepsilon^{\textstyle{*}})\right\}\; \mid Y_1,...,Y_n \right] \\
  &\leq p\tr\left(\widehat{\Sigma}\right)/n 
\end{align*}
where the last inequality follows from the argument that proves 
Lemma~\ref{lem1} applied to the starred data, and 
$p\tr\left(\widehat{\Sigma}\right)/n \to 0$ a.s. It remains 
to show that $\Sigstar_n$ converges to $\Sigma$. Conditional on 
$Y_1,...,Y_n$,
\begin{equation} \label{thm2step}
  \begin{split}
    &d_2^{r(r+1)/2}\left\{n^{-1}\sum_{i=1}^n\text{vech}(\varepsilon_i^{\textstyle{*}}\varepsilon_i^{\textstyle{*}^T}),
      n^{-1}\sum_{i=1}^n\text{vech}(\varepsilon_i\varepsilon_i^T)\right\} \\
    &\qquad\leq d_2^{r(r+1)/2}\left\{\text{vech}(\varepsilon_1^{\textstyle{*}}\varepsilon_1^{\textstyle{*}^T}),
      \text{vech}(\varepsilon_1\varepsilon_1^T)\right\}
  \end{split}
\end{equation}
by Lemma 8.6 in \cite{bickel}. Now $\varepsilon^{\textstyle{*}}$ has 
conditional distribution $\widehat{F}_n$ and $\varepsilon$ has law $F$ and 
Lemma~\ref{lem5} gives $d_2^r\left(\widehat{F}_n,F\right) \to 0$ almost 
everywhere. We now show that 
$d_1\left\{\text{vech}(\varepsilon_1^{\textstyle{*}}\varepsilon_1^{\textstyle{*}^T}),\text{vech}(\varepsilon_1\varepsilon_1^T)\right\} \to 0$ 
a.s. by Lemma 8.5 of \cite{bickel} with 
$\phi(x) = \vechop{xx^T}$ where $x \in \R^r$. To do this, we show that $K$ can 
be chosen so that $\|\phi(x)\|_1 \leq K(1 + \|x\|_2^2)$ where $\|\cdot\|_1$ 
and $\|\cdot\|_2$ are the $\Lnorm^1$ and $\Lnorm^2$ norms respectively. From 
the definition of the Euclidean norm, we have $\|x\|_2^2 = \sum_{i=1}^r x_i^2$. 
It is clear that $x_i^2 + x_j^2 \geq 2|x_ix_j|$ for all $i,j = 1,...,r$. 
Now, pick $K = {r \choose 2}  + 1$. We see that
\begin{align*}
K(1 + \|x\|_2^2) &\geq \abs[\Big]{\sum_{i=1}^rx_i^2 + {r \choose 2}\sum_{i=1}^rx_i^2} \geq \abs[\Big]{\sum_{i=1}^rx_i^2 + \sum_{i\neq j}^r|x_ix_j|} \\
 &\qquad \geq \abs[\Big]{\sum_{i\geq j}^r|x_ix_j|} \geq \|{\vechop{xx^T}}\|_1 = \|\phi(x)\|_1
\end{align*}
A similar argument shows that $1/n\sum_{i=1}^n\varepsilon^{\textstyle{*}}_i$ 
converges to 0. Part c) follows from both a) and b).
\end{proof}


\subsection{Random design and heteroskedasticity}

In this section we provide the proof of Theorem~\ref{MVboot-MVLR-randomX}. 
Several quantities and lemmas are introduced in order to prove
Theorem~\ref{MVboot-MVLR-randomX}. The logic follows that of 
\cite[Section 3]{freedman}. Define,
\begin{align*}
  \Sigma(\mu) &= \int xx^T \mu(dx), \\
  \beta(\mu) &= \int yx^T \mu(dx, dy) \Sigma(\mu)^{-1}, \\
  \varepsilon(\mu,x,y) &= y - \beta(\mu)x^T.
\end{align*}

The next two lemmas are needed to prove Theorem~\ref{MVboot-MVLR-randomX}.

\begin{lem} \label{3.1new}
If $d_4^{p+r}(\mu_n, \mu) \to 0$ as $n \to \infty$, then
\begin{itemize}
  \item[a)] $\Sigma(\mu_n) \to \Sigma(\mu)$ and $\beta(\mu_n) \to \beta(\mu)$, 
  \item[b)] the $\mu_n$-law of $\text{vec}\{\varepsilon(\mu_n,x,y)x^T\}$ 
  converges to the $\mu$-law of $\text{vec}\{\varepsilon(\mu,x,y)x^T\}$ in 
  $d_2^{rp}$,
  \item[c)] the $\mu_n$-law of $\|\varepsilon(\mu_n,x,y)\|^2$ converges to the 
  $\mu$-law of $\|\varepsilon(\mu,x,y)\|^2$ in $d_1$.
\end{itemize}
\end{lem}

\begin{proof}
Part a) immediately follows from \cite[Lemma 8.3c]{bickel}.

We use \cite[Lemma 8.3a]{bickel} to verify part b). The weak convergence 
step is evident. Now,
\begin{align*}
  \|\text{vec}\{\varepsilon(\mu_n,x,y)x^T\}\|^2 &= 
    \|\text{vec}\{yx^T - \beta(\mu_n)xx^T\}\|^2 \\
  &= \|\text{vec}(yx^T)\|^2 + \|\text{vec}(\beta(\mu_n)xx^T)\|^2 
    -2\text{vec}(yx^T)^T\text{vec}\{\beta(\mu_n)xx^T\}.
\end{align*}
Let $z = (x^T,y^T)^T$. Part b) follows from, integration with respect to 
$\mu_n$, part a), and \cite[Lemma 8.5]{bickel} with 
$\phi(z) = \text{vech}(zz^T)$. The steps involving \cite[Lemma 8.5]{bickel} 
are similar to those in the proof of Theorem~\ref{MVboot-MVLR}.

Part c) follows from the same argument used to prove part b).
\end{proof}

\begin{lem} \label{3.2new}
$d_4^{p+r}(\mu_n,\mu) \to 0$ a.e. as $n\to\infty$.
\end{lem}

\begin{proof}
The steps are the same as those in \cite[Lemma 3.2]{freedman}.
\end{proof}

The proof of Theorem~\ref{MVboot-MVLR-randomX} is now given.

\begin{proof}
We can write 
\begin{align*}
  \text{vec}\left\{\sqrt{n}\left(\betastar - \betahat\right)\right\} &= 
    \text{vec}\left[\sqrt{n}
      \left\{\YmatstarT\Xmatstar(\XmatstarT\Xmatstar)^{-1} - \betahat
    \right\}\right] \\
  &=  \text{vec}\left[\sqrt{n}
      \left\{(\epstar + \Xmatstar\betahat^T)^T
        \Xmatstar(\XmatstarT\Xmatstar)^{-1} - \betahat
    \right\}\right] \\
  &= \text{vec}\left\{n^{-1/2}\epstarT\Xmatstar
    (n^{-1}\XmatstarT\Xmatstar)^{-1}\right\} \\
  &= \text{vec}\left(\Zstar\Wstarinv\right) 
    = \left(\Wstarinv \otimes I_r\right)\text{vec}(\Zstar)
\end{align*}
where $\Zstar = n^{-1/2}\epstarT\Xmatstar$ 
and $\Wstar = n^{-1}\XmatstarT\Xmatstar$. \cite[Theorem 3.1]{freedman} 
implies that the conditional law, conditional on $(X_i,Y_i)$, $i=1,...,n$, 
of $\Wstar \to_p \Sigma_X$. This verifies part a).

We now verify part b). From \cite[Lemma 8.7]{bickel}, we have 
$$
  d_2^{rp}\left\{\text{vec}(\Zstar), \text{vec}(Z)\right\}^2 
  \leq 
  d_2^{rp}\left\{\text{vec}(\Xstar_i\epstarT_i), 
    \text{vec}(X_i\varepsilon_i^T)\right\}^2
$$
where the right side goes to 0 a.e. as $n\to\infty$. Lemma~\ref{3.2new} states 
that $\mu_n \to \mu$ a.e. in $d_4^{r+p}$ as $n\to\infty$ and 
part b) of Lemma~\ref{3.1new} implies that the distribution of 
$\text{vec}(\Zstar)$, conditional on $(X_i,Y_i)$, $i=1,...,n$, converges to 
$\text{vec}(Z)$. The random variable $\text{vec}(Z)$ is normally distributed 
with mean 0 and variance matrix $M$. Combining this with part a) verifies that 
the conditional distribution of
$
  \left(\Wstarinv \otimes I_r\right)\text{vec}(\Zstar) 
$  
converges to
$
  \left(\Sigma_X^{-1} \otimes I_r\right)\text{vec}(Z)
$
as $n\to\infty$. This completes the proof of part b).

Part c) follows from the same argument in the proof of 
Theorem~\ref{MVboot-MVLR} where Lemmas~\ref{3.2new} and~\ref{3.1new}c combine 
to show that \eqref{thm2step} converges to 0 as $n\to\infty$. Note that 
$
  \varepsilon_1^{\textstyle{*}} 
    = Y_1^{\textstyle{*}} - \hat{\beta}X_1^{\textstyle{*}}
$
in this argument. This completes the proof.
\end{proof}

\section{Acknowledgments}

The author would like to thank Karl Oskar Ekvall, Forrest Crawford, Snigdhansu 
Chatterjee, Dennis Cook, and two anonymous referees for providing valuable 
feedback which led to the strengthening of this article.


\bibliographystyle{plainnat}

\bibliography{MVbootsources}

\end{document}